\documentclass[reqno, 12pt, reqno]{amsart}

\usepackage{amsfonts, amsthm, amsmath, amssymb}
\usepackage{hyperref}
\hypersetup{colorlinks=false}
\usepackage{comment}

\usepackage[margin=3.4cm]{geometry}

\RequirePackage{mathrsfs}
\let\mathcal\mathscr

\numberwithin{equation}{section}

\newtheorem{theorem}{Theorem}[section]

\newtheorem{proposition}[theorem]{Proposition}

\theoremstyle{definition}

\newtheorem{example}[theorem]{Example}
\newtheorem{remark}[theorem]{Remark}

\newtheorem*{remark*}{Remark}

\newcommand{\BA}{{\mathbb {A}}}

\newcommand{\BC}{{\mathbb {C}}}

\newcommand{\BG}{{\mathbb {G}}}

\newcommand{\BL}{{\mathbb {L}}}

\newcommand{\BN}{{\mathbb {N}}}

\newcommand{\BP}{{\mathbb {P}}}
\newcommand{\BQ}{{\mathbb {Q}}}
\newcommand{\BR}{{\mathbb {R}}}

\newcommand{\BZ}{{\mathbb {Z}}}
\newcommand{\CA}{{\mathcal {A}}}

\newcommand{\CD}{{\mathcal {D}}}

\newcommand{\CF}{{\mathcal {F}}}
\newcommand{\CG}{{\mathcal {G}}}
\newcommand{\CH}{{\mathcal {H}}}
\newcommand{\CI}{{\mathcal {I}}}
\newcommand{\CJ}{{\mathcal {J}}}
\newcommand{\CK}{{\mathcal {K}}}

\newcommand{\CN}{{\mathcal {N}}}
\newcommand{\CO}{{\mathcal {O}}}

\newcommand{\CS}{{\mathcal {S}}}
\newcommand{\CT}{{\mathcal {T}}}
\newcommand{\CU}{{\mathcal {U}}}

\newcommand{\CW}{{\mathcal {W}}}

\newcommand{\bb}{{\mathbf{b}}}

\renewcommand{\phi}{\varphi}
\renewcommand{\rho}{\varrho}
\renewcommand{\epsilon}{\varepsilon}

\newcommand{\bx}{\boldsymbol{x}}

\newcommand{\bt}{\boldsymbol{t}}

\newcommand{\bc}{\boldsymbol{c}}

\newcommand{\blambda}{\boldsymbol{\lambda}}
\newcommand{\bLambda}{\boldsymbol{\Lambda}}
\newcommand{\bGamma}{\boldsymbol{\Gamma}}
\newcommand{\bsigma}{\boldsymbol{\sigma}}

\newcommand{\balpha}{\boldsymbol{\alpha}}

\newcommand{\e}{\textup{e}}

\title[Ternary quadratic forms III]{Quantitative strong approximation for ternary quadratic forms III}
\author{Zhizhong Huang}
\address{State Key Laboratory of Mathematical Sciences, Academy of Mathematics and Systems Science, Chinese Academy of Sciences, Beijing 100190, China}
\email{zhizhong.huang@yahoo.com}
\date{September 2023, Revised December 2025}

\begin{document}
	\begin{abstract}
		We prove asymptotic formulas for counting (primitive) integral points with local conditions on the (punctured) affine cone defined by a non-singular integral ternary quadratic form, and we relate our results to the Brauer--Manin obstruction. Our approach is based on the $\delta$-variant of the Hardy--Littlewood circle method developed by Heath-Brown.
	\end{abstract}
	\maketitle
	\tableofcontents
	\section{Introduction}
	
	In \cite{H-Bdelta}, Heath-Brown developed his $\delta$-circle method and for the first time applied to counting integer solutions of homogeneous ternary quadratic forms. This gave a new analytic proof of the Hasse--Minkowski theorem for ternary quadratic forms which goes back to Legendre (cf. \cite[IV. Theorem 8 (ii)]{Serre}). We begin by recalling his results. 
	
	Let $F(x_1,x_2,x_3)\in\BZ[x_1,x_2,x_3]$ be a non-degenerate indefinite ternary integral quadratic form.  Let $w:\BR^3\to\BR$ be a fixed infinitely differentiable weight function of class $\mathcal{C}(S)$ as in \cite[\S2]{H-Bdelta}, and we define the \emph{weighted singular integral} (with respect to the quadratic form $F$) to be \begin{equation}\label{eq:singint}
		\CI(w):=\iint_{\BR^3\times\BR}w(\bt)\e(\theta F(\bt))\operatorname{d}\bt\operatorname{d}\theta.
	\end{equation} 
	We let $$\widehat{\mathfrak{S}}:=\prod_{p<\infty}\left(1-\frac{1}{p}\right)\sigma_{p},$$ 
	where for every prime $p$, we define the \emph{$p$-adic local density} as
	\begin{equation}\label{eq:sigmap}
		\sigma_{p}:=\lim_{k\to\infty}\frac{\#\{\bx\in(\BZ/p^{k}\BZ)^3: F(\bx)\equiv 0\bmod p^k\}}{p^{2k}}.
	\end{equation}
	This infinite product is absolutely convergent, as shown in \cite{H-Bdelta}.
	
	\begin{theorem}[Heath-Brown \cite{H-Bdelta} Theorem 8, Corollary 2]\label{thm:H-Bmain}
		We have $$\sum_{\bx\in\BZ^3:F(\bx)=0}w\left(\frac{\bx}{B}\right)=\frac{1}{2}\CI(w)\widehat{\mathfrak{S}} B\log B+\sigma_1(w)B+O_\varepsilon(B^{\frac{5}{6}+\varepsilon}),$$
		$$\sum_{\bx\in\BZ_{\operatorname{prim}}^3:F(\bx)=0}w\left(\frac{\bx}{B}\right)=\frac{1}{2}\CI(w)\widehat{\mathfrak{S}} B+O\left(B\exp\left(-c_0(\log B)^\frac{1}{2}\right)\right).$$
	\end{theorem}
	
	However, it is remarked in \cite[p. 161]{H-Bdelta} that there can exist local obstructions for (primitive) integral solutions. Adding congruence conditions into the counting comes with some technical affection, such that the asymptotic formula may not appear as the expected form. 	
	 Based on the approaches developed in \cite{PART1,HuangTernary}, the goal of the present paper is to address this issue more thoroughly, to work out explicit asymptotic formulas, and to relate them to the Brauer--Manin obstruction, in spirit of \cite{PART2}.
	 
	 \bigskip
	 Let $$C:=(F=0)\subset\BP^2$$ which is a projective conic, $$W:=(F=0)\subset \BA^3$$ be the affine cone over $C$, and $$W^o:=(F=0)\subset \BA^3\setminus\boldsymbol{0}$$ be the punctured affine cone. Let $\mathcal{C},\CW,\CW^o$ be respectively the integral models for them also defined by $F$. Zero solutions of $F$ corresponds to integral points on $\CW$, whereas primitive zero solutions can be viewed as integral points on $\CW^o$, which, up to sign, can identified with rational points of $C$.
	
	To state our results, we fix $L\in\BN$ and $\bGamma\in \CW^o(\BZ/L\BZ)$.
	For every prime $p$, we define the \emph{$p$-adic local density} (with respect to the congruence condition $(L,\bGamma)$) to be \begin{equation}\label{eq:sigmapL}
		\sigma_{p}(L,\bGamma):=\lim_{k\to\infty}\frac{\#\{\bx\in(\BZ/p^{k}\BZ)^3: F(\bx)\equiv 0\bmod p^k,\bx\equiv\bGamma\bmod p^{\operatorname{ord}_p(L)}\}}{p^{2k}},
	\end{equation} and we define \emph{the ``modified'' singular series} (with the convergence factors $\left(1-p^{-1}\right)_p$) to be \begin{equation}\label{eq:singser}
	\widehat{\mathfrak{S}}_{L,\bGamma}:=\prod_{p<\infty}\left(1-\frac{1}{p}\right)\sigma_{p}(L,\bGamma).
	\end{equation}
		We define the (weighted) counting functions respectively for $\CW$ and $\CW^o$:
 \begin{equation}\label{eq:CNW}
	\CN_{\CW}(w,(L,\bGamma);B):=\sum_{\substack{\bx\in\BZ^3,F(\bx)=0\\ \bx\equiv \bGamma\bmod L}}w\left(\frac{\bx}{B}\right),
\end{equation}
\begin{equation}\label{eq:CNWo}
		\CN_{\CW^o}(w,(L,\bGamma);B):=\sum_{\substack{\bx\in\BZ_{\text{prim}}^3,F(\bx)=0\\ \bx\equiv \bGamma\bmod L}}w\left(\frac{\bx}{B}\right).
\end{equation}
	
Our main counting results are as follows, which will be established in \S\ref{se:count}.
\begin{theorem}\label{thm:maincountingint}
	There exists $\CH_{L,\bGamma}(w)\in\BC$ such that
	$$\CN_{\CW}(w,(L,\bGamma);B)=\frac{1}{2}\widehat{\mathfrak{S}}_{L,\bGamma}\CI(w)B\log B+\CH_{L,\bGamma}(w)B+O_\varepsilon(B^{\frac{5}{6}+\varepsilon}).$$
\end{theorem}

We shall write $$\Omega:=8L\Delta$$ which contains all ``bad primes''.

	\begin{theorem}\label{thm:mainprimitive}
		Assume that $\Omega\mid L$. There exists $\CG_{L,\bGamma}(w)\in\BC$ depending on $w,(L,\bGamma)$ such that 
		$$\CN_{\CW^o}(w,(L,\bGamma);B)=\left(\frac{1}{2}\widehat{\mathfrak{S}}_{L,\bGamma}\CI(w)+\CG_{L,\bGamma}(w)\right)B+O\left(B\exp\left(-c_0(\log B)^\frac{1}{2}\right)\right),$$ for an appropriate numerical constant $c_0>0$. The explicit expression  of $\CG_{L,\bGamma}(w)$ is given by \eqref{eq:CG}.
	\end{theorem}

	Let us now compare Theorems \ref{thm:maincountingint} \ref{thm:mainprimitive} with Heath-Brown's result (Theorem \ref{thm:H-Bmain}). 
	The common feature is that, without primitivity condition, the order of growth is $B\log B$, the leading constant being the product of local densities with singular integrals. Imposing primitivity condition, via a Möbius inversion, renders $B\log B$ to $B$ and leaves the leading constant unchanged.
	The main difference is that, the secondary term $\CH_{L,\bGamma}(w)B$ becomes $\CG_{L,\bGamma}(w) B$, and it can happen that $\CG_{L,\bGamma}(w)\neq 0$ (it does equal $0$ when no local condition is added), and that $\frac{1}{2}\widehat{\mathfrak{S}}_{L,\bGamma}\CI(w)+\CG_{L,\bGamma}(w)=0$, provided that local obstruction exists. Although we are not quite able to compute the exact value of  $\CG_{L,\bGamma}(w)$, it is expressed as an infinite sum of various exponential sums with respect to certain non-zero Poisson variables and encode the effect of $w$ and non-trivial real characters of modulus dividing $\Omega$.
	 
	 	  
\begin{example} 
	 
	 Heath-Brown \cite[p. 161]{H-Bdelta} addresses briefly the (counter-)example where there is an obstruction to the existence of global primitive integer solutions to the equation $x_1x_2=x_3^2$ with the local conditions $x_i>0$ and $x_i\equiv 2\bmod 3$ for all $i=1,2,3$. Indeed, the primitivity and positivity imply that $\gcd(x_1,x_2)=1$, and that both $x_1$ and $x_2$ are perfect squares Hence they all have the same congruence class $1\bmod 3$, contradiction. The obstruction comes from ``reciprocity'' between the real sign and the Legendre symbol modulo $3$. We shall see that this is coherent with our asymptotic formula.
	\end{example}
	
	In \S\ref{se:puncturedtoprojective}, we deduce a corresponding counting result for the projective conic $C$, based on Theorem \ref{thm:mainprimitive}. This amounts to adding up contributions from all representives of a fixed projective congruence residue. Although $C\simeq\BP^1$, and counting points on $\BP^1$, which is clearly unobstructed, can be done much more elementarily, our goal is to explain how obscured term $\CG_{L,\bGamma}(w)$ disappears on the one hand. On the other hand, we give an interpretation of the $\frac{1}{2}$-factor in Theorem \ref{thm:mainprimitive}.
	
	As a $\BG_{\operatorname{m}}$-torsor over the conic $C$,  the quasi-affine variety $W^o$ is a (smooth) toric variety. That the Brauer--Manin obstruction is the only one to strong approximation off infinity for such varieties has been studied in depth by work of Cao--Xu \cite{Cao-Xu}. We compare our counting results in view of Brauer--Manin obstruction in \S\ref{se:BM}.
		 	 
		 	 \bigskip
	Most of our notation is standard, and is adapted from \cite{PART1,PART2}.
	We write $$\e(x):=\exp\left(2\pi i x\right),\quad \e_q(x):=\exp\left(\frac{2\pi i x}{q}\right).$$ Let $$\phi(n):=\#\left(\BZ/n\BZ\right)^\times$$ denote the Euler totient function, and $\mu$ denote the Möbius function.
	
	 \section{Counting integral points with local conditions}\label{se:count}
	 \subsection{Applying the $\delta$-method}
	 	We choose a representative $\blambda\in\BZ^3$ of $\bGamma$ of size $\ll L$. 
	 We execute the $\delta$-method as in \cite[\S2.4]{PART1} and get 
	 \begin{equation}\label{eq:Nwdelta}
	 	\begin{split}
	 		\CN_{\CW}(w,(L,\bGamma);B)=\frac{C_Q}{Q^2}\left(\frac{B}{L^2}\right)^3\sum_{q=1}^{\infty}\sum_{\bc\in\BZ^{3}}\frac{\e_{qL^2}(\bc\cdot\blambda)S_{q,L,\blambda}(\bc)\widehat{\CI}_{\frac{q}{Q}}\left(w;\frac{\bc}{L}\right)}{q^{3}}, 	\end{split}
	 \end{equation} where $C_Q=1+O_N(Q^{-N})$ and for every $q\ll Q, \bc\in\BZ^3$,
	 we let \begin{equation}\label{eq:Sqc}
	 	S_{q,L,\blambda}(\bc):=\sum_{\substack{a\bmod q\\(a,q)=1}}\sum_{\substack{\bsigma\in(\BZ/qL\BZ)^3\\  L^2\mid F(L\bsigma+\blambda_N)}}\textup{e}_{qL}\left(a\left(\frac{F(L\boldsymbol{\sigma}+\blambda_N)}{L}\right)+\bc\cdot\bsigma\right),
	 \end{equation}
	 and for any $\bb\in\BR^3$ and $r\in\BR_{>0}$ we define
	 \begin{equation}\label{eq:JqL}
	 	\CI_{r}(w;\bb):=\int_{\BR^3}w(\bt)h\left(r,F(\bt)\right)\e_r\left(-\bb\cdot\bt\right)\operatorname{d}\bt.
	 \end{equation}
	 
	Write \begin{equation}\label{eq:qdecomp}
		q=q_1q_2
	\end{equation} such that $$\gcd(q_1,\Omega)=1,\quad q_2\mid \Omega^\infty.$$ Write $$\frac{F(\blambda)}{L}=k_2q_1+k_1q_2L,$$ with $k_1\bmod q_1$ and $k_2\bmod q_2L$.
	 	By the Chinese remainder theorem as in \cite[Lemma 4.2]{PART1}, we have	$$S_{q,L,\blambda}(\bc)=S^{(1)}_{q,L,\blambda}(\bc)S^{(2)}_{q,L,\blambda}(\bc),$$ where $S^{(1)}_{q,L,\blambda}(\bc)$ (resp. $S^{(2)}_{q,L,\blambda}(\bc)$) is a quadratic exponential sum modulo integers coprime to $\Omega$ (resp. whose prime factors all dividing $\Omega$).

	 By \cite[Proposition 4.3]{PART1}, we have 
	 $$S^{(1)}_{q,L,\blambda}(\bc)=\e_{q_1}\left(-\overline{q_2L^2}\bc\cdot\blambda\right)\iota_{q_1}^{3}\CT(q_1;\bc),$$ where for every $x\in\BN$, we write $\iota_{x}$ for the quadratic Gauss sum of modulus $x$ if $x$ is odd (cf. \cite[Lemma 4.3]{PART1}), and we let \begin{equation}\label{eq:CT}
	 	\CT(x;\bc):=\sum_{\substack{a\bmod x\\(a,x)=1}}\left(\frac{a}{x}\right)\e_{x}\left(-aF^*(\bc)\right),
	 \end{equation}	 $F^*$ standing for the dual of $F$.
	  
	  As for $S^{(2)}$, we have $$S^{(2)}_{q,L,\blambda}(\bc)= \e_{q_2L^2}\left(-\overline{q_1}\bc\cdot\blambda\right)\CS_{q_2,L,\blambda}(q_1;\bc),$$ where for any $x\in\BN$ such that $\gcd(x,L)=1$, we let $$\CS_{q_2,L,\blambda}(x;\bc):=\sum_{\substack{a_2\bmod q_2\\(a_2,q)=1}}\sum_{\substack{\balpha\in (\BZ/q_2L^2\BZ)^3\\F(\balpha)\equiv 0\bmod L^2\\ \balpha\equiv \overline{x}\blambda\bmod L}}\e_{q_2L^2}\left(a_2 F(\balpha)+\bc\cdot\balpha\right).$$
	 Moreover by \cite[Proposition 4.6]{PART1} \begin{equation}\label{eq:S2bd}
	 	S^{(2)}_{q,L,\blambda}(\bc)\ll q_2^\frac{5}{2}.
	 \end{equation}
	 
	 \subsection{Contribution from $\bc\neq\boldsymbol{0}$}
	 We next derive the contribution from non-zero Poisson variables $\bc$ in \eqref{eq:Nwdelta}.
	 	 \begin{proposition}\label{prop:cneq0}
There exists $\CK_{L,\bGamma}(w)\in\BC$ and $0<\iota<1$ such that	$$\sum_{q=1}^{\infty}\sum_{\bc\in\BZ^{3}\setminus\boldsymbol{0}}\frac{\e_{qL^2}(\bc\cdot\blambda)S_{q,L,\blambda}(\bc)\CI_{\frac{q}{Q}}\left(w;\frac{\bc}{L}\right)}{q^{3}}=L^4\CK_{L,\bGamma}(w)+O_\varepsilon(B^{-\iota+\varepsilon}).$$
	 The expression of $\CK_{L,\bGamma}(w)$ is given by \eqref{eq:CK}.
	 \end{proposition}

	 \subsubsection{}
For every $\bc\neq\boldsymbol{0}$ and $X\gg 1$, we now let
	 \begin{equation}
	 		\CF_{\blambda}(\bc;X):=\sum_{\substack{q\leqslant X}}\frac{\e_{qL^2}(\bc\cdot\blambda)S_{q,L,\blambda}(\bc)}{q^2}.
	 \end{equation}
	 \begin{proposition}\label{prop:qsumcneq0}
	We have 	$$\CF_{\blambda}(\bc;X)=\eta_{\blambda}(\bc)X+O_\varepsilon(|\bc|X^{\frac{1}{2}+\varepsilon}),$$ where $\eta_{\blambda}(\bc)$ is defined in \eqref{eq:eta1} and \eqref{eq:eta2} and satisfies $\eta_{\blambda}(\bc)\ll_\varepsilon |\bc|^\varepsilon$.
	 \end{proposition}
	 \begin{proof}[Proof of Proposition \ref{prop:qsumcneq0}]
	 	For each fixed $q$ and $\bc\in\BZ^3$, for every Dirichlet character $\chi\bmod L$ let  \begin{equation}\label{eq:CAchic}
	 		\CA_{q_2,L,\blambda}(\chi;\bc):=\frac{1}{\phi(L)}\sum_{x\bmod L}\chi(x)^c\CS_{q_2,L,\blambda}(x;\bc).
	 	\end{equation} 
	 	Then $$\CS_{q_2,L,\blambda}(q_1;\bc)=\sum_{\chi\bmod L}\chi(q_1)\CA_{q_2,L,\blambda}(\chi;\bc).$$
	 	So we can decompose 
	 	$$\CF_{\blambda}(\bc;X)=\sum_{\substack{q_2\leqslant X\\ q_2\mid\Omega^\infty}}\frac{1}{q_2^2}\sum_{\chi\bmod L}\CA_{q_2,L,\blambda}(\chi;\bc)\CU_{\bc,\chi}\left(\frac{X}{q_2}\right),$$ where
	 	$$\CU_{\bc,\chi}\left(Y\right):=\sum_{\substack{x\leqslant Y, (x,\Omega)=1}}\frac{\chi(x)\CT(x;\bc)}{\iota_{x}},$$ $\CT(x;\bc)$ being defined by \eqref{eq:CT}.
	 	
	 	\emph{Case where $F^*(\bc)=0$.} Then
	 	$$\CT(x;\bc)=\sum_{a\bmod x}\left(\frac{a}{x}\right)=\begin{cases}
	 		\phi(x) & \text{ if } x=\square;\\ 0 &\text{ otherwise}.
	 	\end{cases}$$
	 	So in this case ($\iota_{q_1}=\sqrt{q_1}$ if $q_1=\square$ and is odd by \cite[(3.38)]{Iwaniec-Kolwalski})
	 	$$\CU_{\bc,\chi}\left(Y\right)=\sum_{\substack{q_1\leqslant Y\\q_1=\square, (q_1,\Omega)=1}}\frac{\chi(q_1)\phi(q_1)}{\sqrt{q_1}}=\sum_{\substack{r\leqslant\sqrt{Y}\\(r,\Omega)=1}}\frac{\phi(r^2)\chi^2(r)}{r}.$$
	 	Assuming that $\chi$ is non-real, i.e., $\chi^2$ is non-principal, we have, by the trivial bound for character sums (view $\chi$ as modulo $\Omega$) coupled with partial summation,
	 	\begin{align*}
	 		\CU_{\bc,\chi}\left(Y\right)&=\sum_{\substack{r\leqslant \sqrt{Y}}}\chi^2(r)r\sum_{d\mid r}\frac{\mu(d)}{d}\\ &=\sum_{\substack{e_1\leqslant\sqrt{Y}}}\mu(e_1)\chi^2(e_1)\sum_{e_2\leqslant\frac{\sqrt{Y}}{e_1}}\chi^2(e_2)e_2\ll_\varepsilon Y^{\frac{1}{2}+\varepsilon}.
	 	\end{align*}
	 	Now assume that $\chi$ is real. Then by the well-known asymptotic for $\phi$, we have
	 	\begin{align*}
	 		\CU_{\bc,\chi}\left(Y\right)&=\sum_{\substack{r\leqslant \sqrt{Y}\\ (r,\Omega)=1}}\phi(r)=\frac{3}{\pi^2}\prod_{p\mid \Omega}\left(1-\frac{1}{p}\right)Y+O_\varepsilon(Y^{\frac{1}{2}+\varepsilon}).
	 	\end{align*}
	 	
	 	We now go back to $\CF_{\blambda}(\bc;X)$. Note that \eqref{eq:S2bd} implies that $$\CA_{q_2,L,\blambda}(\chi;\bc)\ll q_2^\frac{5}{2}.$$
	 	We obtain that 
	 	\begin{align*}
	 		\CF_{\blambda}(\bc;X)&=\frac{3}{\pi^2}\prod_{p\mid \Omega}\left(1-\frac{1}{p}\right)X\sum_{\substack{q_2\leqslant X\\ q_2\mid\Omega^\infty}}\sum_{\substack{\chi\bmod L\\ \chi \text{ real}}}\frac{\CA_{q_2,L,\blambda}(\chi;\bc)}{q_2^3}+O_\varepsilon(X^{\frac{1}{2}+\varepsilon})\\ &=\eta_{\blambda}(\bc)X+O_\varepsilon(X^{\frac{1}{2}+\varepsilon}),
	 	\end{align*}
 where \begin{equation}\label{eq:eta1}
	 		\eta_{\blambda}(\bc):=\sum_{\substack{\chi\bmod L\\ \chi \text{ real}}}\eta_{\blambda}(\chi;\bc)
	 	\end{equation} and for $F^*(\bc)=0$ and for real $\chi\bmod L$ we define
	 	\begin{equation}\label{eq:etachic1}
	 		\eta_{\blambda}(\chi;\bc):=\frac{3}{\pi^2}\prod_{p\mid \Omega}\left(1-\frac{1}{p}\right)\sum_{u\mid \Omega^\infty}\frac{\CA_{u,L,\blambda}(\chi;\bc)}{u^3}.
	 	\end{equation}
	 	Clearly $$\eta_{\blambda}(\bc)\ll 1.$$

	\emph{Case where $F^*(\bc)\neq 0$.} 
	For any odd integer $q$, we factorize $q=q'q''$ where $q'$ is square-free, $q''$ is square-full and $\gcd(q',q'')=1$. Then we have (see e.g. \cite[(3.28) (3.38)]{Iwaniec-Kolwalski}) $$\iota_q=\left(\frac{q''}{q'}\right)\left(\frac{q'}{q''}\right)\iota_{q'}\iota_{q''},$$ since $q$ is odd.
	Now by e.g. \cite[(3.12)]{Iwaniec-Kolwalski}, \begin{align*}
		\CT(q;\bc) &=\sum_{a'\bmod q'}^{*}\sum_{a''\bmod q''}^{*}\left(\frac{a'}{q'}\right)\left(\frac{a''}{q''}\right)\e_{q'}\left(-\overline{q''}a'F^*(\bc)\right)\e_{q''}\left(-\overline{q'}a''F^*(\bc)\right)\\ &=\left(\frac{q''}{q'}\right)\left(\frac{q'}{q''}\right)\CT(q';\bc)\CT(q'';\bc)\\ &=\left(\frac{q''}{q'}\right)\left(\frac{q'}{q''}\right)\left(\frac{-F^*(\bc)}{q'}\right)\iota_{q'}\CT(q'';\bc)\\ &=\iota_q\left(\frac{-F^*(\bc)}{q'}\right)\frac{\CT(q'';\bc)}{\iota_{q''}}.
	\end{align*}  

It follows that
\begin{align*}
	\CU_{\bc,\chi}\left(Y\right)&=\sum_{\substack{q''\leqslant Y,\square-\text{full}\\ \gcd(q'',\Omega)=1}}\chi(q'')\frac{\CT(q'';\bc)}{\iota_{q''}}\sum_{\substack{q'\leqslant\frac{Y}{q''},\square-\text{free}\\ \gcd(q',q''\Omega)=1}}\chi(q')\left(\frac{-F^*(\bc)}{q'}\right).
\end{align*}
Let $$\chi_{\bc}(\cdot):=\chi(\cdot)\left(\frac{-F^*(\bc)}{\cdot}\right),$$ viewed as a character of modulus $\Omega |F^*(\bc)|\ll |\bc|^2$. We note that $\chi_{\bc}$ is principal if and only if 
\begin{equation}\label{eq:creal}
	\text{There exists an integer } \xi_{\bc}\mid L \text{ such that } \chi(\cdot)=\left(\frac{\xi_{\bc}}{\cdot}\right) \text{ and } -F^*(\bc)=\xi_{\bc}\square.
\end{equation}

If $\chi_{\bc}$ is non-principal, then using the identity $\mu^2(n)=\sum_{d^2\mid n}\mu(d)$ and by Burgess's $\frac{3}{16}$-bound \cite{Burgess}, the inner sum is \begin{align*}
	&=\sum_{d''\mid q''}\mu(d'')\chi_{\bc}(d'')\sum_{\substack{q'\leqslant \frac{Y}{q''d''}}}\sum_{d^{\prime 2}\mid q'}\mu(d')\chi_{\bc}(q')\\
	&=\sum_{d''\mid q''}\mu(d'')\chi_{\bc}(d'')\sum_{d'\leqslant\left(\frac{Y}{q''d''}\right)^\frac{1}{2}}\mu(d')\chi_{\bc}(d'^2)\sum_{\substack{q'\leqslant \frac{Y}{q''d''d'^2}}}\chi_{\bc}(q')\\
&\ll_\varepsilon |\bc|^{\frac{3}{8}+\varepsilon}\left(\frac{Y}{q''}\right)^{\frac{1}{2}+\varepsilon}\tau(q'')\sum_{d'\leqslant \left(\frac{Y}{q''}\right)^\frac{1}{2}}\frac{1}{d'}\ll_\varepsilon |\bc|^{\frac{3}{8}+\varepsilon}\left(\frac{Y}{q''}\right)^{\frac{1}{2}+\varepsilon}.
\end{align*}
Using explicit evaluation of Gauss sums (see e.g. \cite[Lemma 3.2]{Iwaniec-Kolwalski} or \cite[Proof of Corollary 4.5]{PART1}), we have \begin{equation}\label{eq:CTqbd}
	\CT(q'';\bc)\ll_\varepsilon |\bc| q^{''\frac{1}{2}+\varepsilon}.
\end{equation} Hence 
\begin{align*}
	\CU_{\bc,\chi}\left(Y\right)\ll_\varepsilon&|\bc|^{\frac{11}{8}+\varepsilon}Y^{\frac{1}{2}+\varepsilon}\sum_{\substack{q''\leqslant Y,\square-\text{full}}}q^{-\frac{1}{2}+\varepsilon}\ll_\varepsilon  |\bc|^{\frac{
			11}{8}+\varepsilon} Y^{\frac{7}{8}+\varepsilon}.
\end{align*}

If $\chi_{\bc}$ is principal, by the well-known asymptotic for square-free integers, the inner sum equals $$\frac{Y}{q''}\frac{6}{\pi^2}\prod_{p\mid q''\Omega F^*(\bc)}\left(1-\frac{1}{p}\right)+O_\varepsilon\left(Y^{\frac{1}{2}+\varepsilon}q^{-\frac{1}{2}+\varepsilon}|\bc|^\varepsilon\right).$$ Hence, using again \eqref{eq:CTqbd}, 
\begin{align*}
	\CU_{\bc,\chi}\left(Y\right)&=Y\frac{6}{\pi^2}\sum_{\substack{q''\leqslant Y,\square-\text{full}}}\frac{\chi(q'')\CT(q'';\bc)}{q''\iota_{q''}}\prod_{p\mid q''\Omega F^*(\bc)}\left(1-\frac{1}{p}\right)+O_\varepsilon\left(|\bc|^\varepsilon Y^{\frac{7}{8}+\varepsilon}\right)\\ &=Y\theta_{\chi}(\bc)+O_\varepsilon\left(|\bc| Y^{\frac{7}{8}+\varepsilon}\right).
\end{align*}
 where we define, for $\chi,\bc$ satisfying \eqref{eq:creal}, \begin{equation}
	\theta_{\chi}(\bc):=\frac{6}{\pi^2}\sum_{\substack{x~\square-\text{full}\\\gcd(x,\Omega)=1}}\frac{\chi(x)\CT(x;\bc)}{x\iota_{x}}\prod_{p\mid x\Omega F^*(\bc)}\left(1-\frac{1}{p}\right),
\end{equation}
and we have $$\theta_{\chi}(\bc)\ll\sum_{\substack{x~\square-\text{full}}}\frac{\tau(x)x^\frac{1}{2}\gcd(x,F^*(\bc))^\frac{1}{2}}{x^\frac{3}{2}}\ll_\varepsilon |\bc|^\varepsilon.$$

Going back to $\CF_{\blambda}(\bc;X)$ we have
	 	$$\CF_{\blambda}(\bc;X)=\eta_{\blambda}(\bc)X+O_\varepsilon(|\bc|X^{\frac{7}{8}+\varepsilon}),$$ where for $F^*(\bc)\neq 0$ we define
\begin{equation}\label{eq:eta2}
			\eta_{\blambda}(\bc):=\sum_{\substack{\chi\bmod L\\ \eqref{eq:creal}\text{ holds}}}\eta_{\blambda}(\chi;\bc),
\end{equation}
where \begin{equation}\label{eq:etachic2}
	\eta_{\blambda}(\chi;\bc):=\sum_{u\mid \Omega^\infty}\frac{\CA_{u,L,\blambda}\left(\chi;\bc\right)\theta_{\chi}(\bc)}{u^3}.
\end{equation}
Clearly  \begin{equation*}
	\eta_{\blambda}(\bc)\ll_\varepsilon |\bc|^\varepsilon.\qedhere
\end{equation*}
	 \end{proof}


	\subsubsection{Proof of Proposition \ref{prop:cneq0}}
	Let us define, as in \cite[(5.1)]{HuangTernary},
	 $$\CJ_w(\bc):=\int_{0}^\infty\frac{\CI_{r}(w;\frac{\bc}{L})}{r}\operatorname{d}r.$$ 
	 The arguments of \cite[Proof of Theorem 5.2]{HuangTernary} show that 
	 $$\CJ_w(\bc)\ll_N |\bc|^{-N},$$
	 and moreover, with the replacement of \cite[Theorem 4.1 (2)]{HuangTernary} by Proposition \ref{prop:qsumcneq0}, lead to the asymptotic formula in Proposition \ref{prop:cneq0} with the (well-defined) constant \begin{equation}\label{eq:CK}
	 	\CK_{L,\bGamma}(w):=\frac{1}{L^4}\sum_{\substack{\chi\bmod L\\ \chi\text{ real}}}\sum_{\substack{\substack{\bc\in\BZ^3\setminus\boldsymbol{0}\\\text{either } F^*(\bc)=0\\ \text{or } -F^*(\bc)=|\chi|\square\neq 0}}}\eta_{\blambda}(\chi;\bc)\CJ_w(\bc).
	 \end{equation}	
	 in the main term, and a power-saving term $O_\varepsilon(B^{-\iota+\varepsilon})$ with a certain numerical constant $0<\iota<1$. \qed

	 \subsection{Contribution from $\bc=\boldsymbol{0}$} 
	 We recall \eqref{eq:singint} \eqref{eq:singser}.
	 \begin{proposition}\label{prop:c=0}
	 	There exists $b_{L,\bGamma}(w)\in\BC$ such that 
	 	$$\sum_{q=1}^{\infty}\frac{S_{q,L,\blambda}(\boldsymbol{0})\CI_{\frac{q}{Q}}\left(w;\boldsymbol{0}\right)}{q^{3}}=\widehat{\mathfrak{S}}_{L,\bGamma}L^4\left(\frac{1}{2}\CI(w)\log B+b_{L,\bGamma}(w)\right)+O_\varepsilon(B^{-\frac{1}{6}+\varepsilon}).$$
	 \end{proposition}
 \begin{proof}
 Let us consider the formal series
	 $$\varPi(s):=\sum_{q=1}^{\infty}\frac{S_{q,L,\blambda}(\boldsymbol{0})}{q^s}=\prod_{p<\infty}\sum_{t=0}^{\infty}\frac{S_{p^t,L,\blambda}(\boldsymbol{0})}{p^{ts}}.$$
	 For $p\nmid \Omega$, we have
	 $$S_{p,L,\blambda}(\boldsymbol{0})=0,\quad S_{p^2,L,\blambda}(\boldsymbol{0})=p^5\left(1-\frac{1}{p}\right).$$
	 	Therefore $$\varPi(s)=\zeta(2s-5)\nu(s)$$ for a certain $\nu(s)$ holomorphic in $\Re(s)>\frac{17}{6}$.
	 Using Perron's formula as in the argument of \cite[Lemmas 30\&31]{H-Bdelta} shows that 
	 $$\sum_{q\leqslant X}S_{q,L,\blambda}(\boldsymbol{0})=\frac{1}{6}L^4\widehat{\mathfrak{S}}_{L,\bGamma}X^3+O_\varepsilon(X^{\frac{17}{6}+\varepsilon}),$$
	 $$\sum_{q\leqslant X}\frac{S_{q,L,\blambda}(\boldsymbol{0})}{q^3}=\frac{1}{2}\widehat{\mathfrak{S}}_{L,\bGamma}\left(L^4\log X+2\gamma\right).$$
	 We emphasize that the residue is taken at e.g. $s=0$ of $\zeta(2s+1)$. That's why the factor $\frac{1}{2}$ appears. (See Remark \ref{rmk:1/2} for an interpretation.) Using \cite[Lemma 13]{H-Bdelta} (see also \cite[Lemma 6.2]{PART1}) and arguing as in \cite[p. 202--p. 203]{H-Bdelta} (see also \cite[\S5.2, (5.11)]{HuangTernary}), we find that
	 $$b_{L,\bGamma}(w)=\widehat{\mathfrak{S}}_{L,\bGamma}\left(\CI(w)(L^{-4}\gamma)+K_w\right),$$ where $K_w\in\BR$ depends on $w$. In particular for any $\gcd(d,L)=1$,
\begin{equation}\label{eq:bflip}
		 b_{L,d\bGamma}=b_{L,\bGamma}.\qedhere
\end{equation}
 \end{proof}

\subsection{Proof of Theorem \ref{thm:maincountingint}}
Recalling \eqref{eq:Nwdelta}, this is a direct consequence of Propositions \ref{prop:cneq0} \ref{prop:c=0}, on setting $$\CH_{L,\bGamma}(w):=b_{L,\bGamma}(w)+\CK_{L,\bGamma}(w).$$  
The proof is completed. \qed

	 \subsection{Proof of Theorem \ref{thm:mainprimitive}}
	 We are ready to deduce our main counting result, under the assumption that $\Omega\mid L$. Our argument follows closely \cite[p. 204--205]{H-Bdelta}.
	 
	 Recall \eqref{eq:sigmapL}. It is clear that for any $p\nmid L$, $\sigma_{p}(L,\bGamma)=\sigma_{p}$ which is the usual $p$-adic density of $\CW$.  Let us define the corresponding $p$-density for $\CW^o$: $$\sigma_{p}^{o}(L,\bGamma) :=\lim_{k\to\infty}\frac{\#\{\bx\in(\BZ/p^{k}\BZ)^3: p\nmid \bx,F(\bx)\equiv 0\bmod p^k,\bx\equiv\bGamma\bmod p^{\operatorname{ord}_p(L)}\}}{p^{2k}}.$$ For $p\nmid L$, we write $$\sigma_{p}^{o}=\sigma_{p}^{o}(L,\bGamma),$$ and we have that $$\left(1-\frac{1}{p}\right)\sigma_{p}=\sigma_{p}^{o}.$$
	 For $p\mid L$, it follows from the expression \eqref{eq:sigmapL} that for every $d$ with $\gcd(d,L)=1$ (since $\bGamma\in\CW^o(\BZ/L\BZ)$, $p\nmid \bGamma$), \begin{align*}
	 	\sigma_{p}(L,d\bGamma)=\sigma_{p}(L,\bGamma)=\sigma_{p}^{o}(L,\bGamma),
	 \end{align*} whence
	  $$\widehat{\mathfrak{S}}_{L,\overline{d}\bGamma}=\widehat{\mathfrak{S}}_{L,\bGamma}$$ for any $d$ coprime with $L$. 
	
	 We shall make use of the formulas (derived from the prime number theorem)
	 $$\sum_{\substack{d\leqslant D\\\gcd(d,L)=1}}\mu(d)\ll D\exp\left(-c_0(\log D)^\frac{1}{2}\right).$$
	 $$\sum_{\substack{d\leqslant D\\\gcd(d,L)=1}}\frac{\mu(d)}{d}\ll \exp\left(-c_1(\log D)^\frac{1}{2}\right),$$
	 $$\sum_{\substack{d\leqslant D\\\gcd(d,L)=1}}\frac{\mu(d)\log d}{d}=-1+O\left(\exp\left(-c_2(\log D)^\frac{1}{2}\right)\right),$$ and for any non-principal real character $\chi\bmod L$,
	 $$\sum_{\substack{d\leqslant D\\\gcd(d,L)=1}}\frac{\mu(d)\chi(d)}{d}=\BL(1,\chi)^{-1}+O\left( \exp\left(-c_3(\log D)^\frac{1}{2}\right)\right),$$
	 where $c_0,c_1,c_2,c_3>0$ are fixed numerical constants.
	 
	 The counting functions \eqref{eq:CNW} \eqref{eq:CNWo} are related by a Möbius inversion:
	 $$\CN_{\CW^o}(w,(L,\bGamma);B)=\sum_{\substack{d\ll B\\ \gcd(d,L)=1}}\mu(d)\CN_{\CW}\left(w,(L,\overline{d}\bGamma);\frac{B}{d}\right).$$ On choosing the truncation factor $D=\sqrt{B}$, summing over the terms obtained from Theorem \ref{thm:maincountingint} of order $B\log B$, we obtain
	 \begin{align*}
	  &\frac{1}{2}\widehat{\mathfrak{S}}_{L,\bGamma}\CI(w)\sum_{\substack{d\leqslant \sqrt{B}\\\gcd(d,L)=1}}\mu(d)\frac{B}{d}\log\left(\frac{B}{d}\right)\\ =&\frac{1}{2}\widehat{\mathfrak{S}}_{L,\bGamma}\CI(w)B+O\left(B\log B\exp\left(-c_4(\log B)^\frac{1}{2}\right)\right).
	 \end{align*}
 Using \eqref{eq:bflip}, summing over the terms involving $b_{L,\overline{d}\bGamma}$ contributes only $O\left(B\exp\left(-c_5(\log B)^\frac{1}{2}\right)\right)$. 

We next focus on the terms with $\CK_{L,\bGamma}(w)$ \eqref{eq:CK}. 
	 We record the following ``flipping formula'' \cite[Lemma 4.2]{PART2}: For every $d\in\BN$ with $\gcd(d,L)=1$, $q\mid L^\infty$, and $\chi\bmod L$,
\begin{equation}\label{eq:flip}
	\CA_{q,L,d\blambda}(\chi;\bc)=\chi(d)^c \CA_{q,L,\blambda}(\chi;\bc).
\end{equation}
Recalling \eqref{eq:etachic1} \eqref{eq:etachic2}, by \eqref{eq:flip}, for any admissible $\bc$ and real $\chi$ (i.e either $F^*(\bc)=0$, or $-F^*(\bc)\neq 0$ and \eqref{eq:creal} holds)
\begin{equation}\label{eq:etaflip}
	\eta_{\overline{d}\blambda}(\chi;\bc)=\chi(d)\eta_{\blambda}(\chi;\bc). 
\end{equation}
Hence $$\sum_{\substack{d\leqslant \sqrt{B}\\\gcd(d,L)=1}}\frac{\mu(d)\CK_{L,\overline{d}\bGamma}(w)}{d}=\CG_{L,\bGamma}(w)+O\left(\exp\left(-c_6(\log B)^\frac{1}{2}\right)\right),$$
where (note that the principal characters have negligible contribution)
\begin{equation}\label{eq:CG}
	\CG_{L,\bGamma}(w):=\frac{1}{L^4}\sum_{\substack{\chi\bmod L\\ \text{non-principal  real}}}\BL(1,\chi)^{-1}\sum_{\substack{\substack{\bc\in\BZ^3\setminus\boldsymbol{0}\\\text{either } F^*(\bc)=0\\ \text{or } -F^*(\bc)=|\chi|\square\neq 0}}}\eta_{\blambda}(\chi;\bc)\CJ_w(\bc).
\end{equation}
The argument of \cite[p. 205]{H-Bdelta} shows that summing over $d>\sqrt{B}$ only adds to an error $O\left(B\exp\left(-c_7(\log B)^\frac{1}{2}\right)\right)$.

Combining everything we obtained so far proves the main theorem. \qed

\section{From the punctured affine cone to the projective conic}\label{se:puncturedtoprojective}
Let $\bLambda\in\mathcal{C}(\BZ/L\BZ)$ be the projective image of $\bGamma\in \CW^o(\BZ/L\BZ)$. Let $\omega_f^C$ be the finite Tamagawa measure on $$C(\mathbf{A}_\BQ^\infty)=\mathcal{C}(\widehat{\BZ}):=\prod_{p<\infty}\mathcal{C}(\BZ_p).$$ The pair $(L,\bLambda)$ specifies a non-archimedean congruence neighbourhood $\CD(L,\bLambda)\subset \mathcal{C}(\widehat{\BZ})$ as in \cite[(1.2)]{PART1}.
We assume that for all $\bx\in\BR^3$,
\begin{equation}\label{eq:wsymmetric}
	w(-\bx)=w(\bx).
\end{equation}
We thus define the counting function $$\CN_{C}(w,(L,\bLambda);B)=\sum_{\substack{P\in C(\BQ)\cap \CD(L,\bLambda)}}w\left(\frac{\bx(P)}{B}\right),$$ where, $\bx(P)\in\BZ_{\text{prim}}^3$ denotes a representative of a rational point $P$ by primitive integer vectors.
\begin{proposition}\label{prop:CNC}
We have, assuming \eqref{eq:wsymmetric},
	$$\CN_{C}(w,(L,\bLambda);B)=\frac{1}{4}\CI(w)\omega_f^C(\CD(L,\bLambda))+O\left(B\exp\left(-c_8(\log B)^\frac{1}{2}\right)\right).$$
\end{proposition}
\begin{proof}
	As in \cite[Proof of Theorem 1.1]{PART1}, we may fix $L_0\in\BN$ such that $\CW^o,\mathcal{C}$ are smooth over $\BZ[\prod_{p\mid L_0}p^{-1}]$ and that for all $p\mid L_0$, $L_0\mid L_1$ $\bGamma_1\in \CW^o(\BZ/L_1\BZ)$, $$\sigma_{p}(L_1,\bGamma_0)=\frac{1}{p^{2\operatorname{ord}_p(L_1)}}.$$
	Moreover, we may assume that $L_0\mid L$.  We then have 
\begin{align*}
	\CN_{C}(w,(L,\bLambda);B)=&\frac{1}{2}\sum_{\gamma\in(\BZ/L\BZ)^\times}\CN_{\CW^o}(w,(L,\gamma\bGamma);B)\\	=&\frac{1}{2}\sum_{\gamma\in(\BZ/L\BZ)^\times}\sum_{\substack{d\ll B\\\gcd(d,L)=1}}\mu(d)\CN_{\CW}\left(w,(L,\gamma\overline{d}\bGamma);\frac{B}{d}\right).
\end{align*}
Using Theorem \ref{thm:mainprimitive}, 
\begin{multline*}
	\CN_{C}(w,(L,\bLambda);B)=\frac{1}{2}\sum_{\gamma\in(\BZ/L\BZ)^\times}\sum_{\substack{d\ll B\\\gcd(d,L)=1}}\frac{\mu(d)}{d}\left(\frac{1}{2}\widehat{\mathfrak{S}}_{L,\gamma\overline{d}\bGamma}\CI(w)+\CG_{L,\gamma\overline{d}\bGamma}(w)\right)B\\+O\left(B\exp\left(-c_0'(\log B)^\frac{1}{2}\right)\right)
\end{multline*}
By \eqref{eq:etaflip}, comparing \eqref{eq:CG}, we have
$$\CG_{L,\gamma\overline{d}\bGamma}(w)=\frac{1}{L^4}\sum_{\substack{\chi\bmod L\\ \text{non-principal  real}}}\chi(\gamma)\chi(d)\BL(1,\chi)^{-1}\sum_{\substack{\substack{\bc\in\BZ^3\setminus\boldsymbol{0}\\\text{either } F^*(\bc)=0\\ \text{or } -F^*(\bc)=|\chi|\square\neq 0}}}\eta_{\blambda}(\chi;\bc)\CJ_w(\bc).$$
We observe that if $\chi\bmod L$ is non-principal,
$$\sum_{\gamma\bmod L}\chi(\gamma)=0,$$ and hence for each fixed $d$, we have $$\sum_{\gamma\bmod L}\CG_{L,\gamma\overline{d}\bGamma}(w)=0.$$ Extending $d$-sum to infinity upon adding an error of the form $O\left(B\exp\left(-c_9(\log B)^\frac{1}{2}\right)\right)$, using $\widehat{\mathfrak{S}}_{L,\gamma\overline{d}\bGamma}=\widehat{\mathfrak{S}}_{L,\bGamma}$, and proceeding in a similar manner as in \cite[(6.17) (6.19) (7.18)]{PART1}, 
\begin{align*}
	\sum_{\gamma\in(\BZ/L\BZ)^\times}\sum_{\gcd(d,L)=1}\frac{\mu(d)}{d}\widehat{\mathfrak{S}}_{L,\gamma\overline{d}\bGamma}&=\phi(L)\prod_{p\mid L}\frac{1}{p^{2\operatorname{ord}_p(L)}}\prod_{p\nmid L}\left(1-\frac{1}{p}\right)\sigma_{p}\\ &=\prod_{p\mid L}\left(1-\frac{1}{p}\right)\frac{1}{p^{\operatorname{ord}_p(L)}}\prod_{p\nmid L}\left(1-\frac{1}{p}\right)\sigma_{p}\\ &=\omega_f^C(\CD(L,\bLambda)).\qedhere
\end{align*}
\end{proof}

\begin{remark}\label{rmk:1/2}
	By \cite[Remark 6.3]{PART1}, the factor $\frac{1}{2}\CI(w)$ equals the (weighted)  Tamagawa volume of the real projective conic $C(\BR)$. However, in stark constrast to all higher dimensional quadrics, Peyre's constant for $C\subseteq\BP^2$ equals $\frac{1}{2}$ (instead of $1$, cf. \cite[(6.5)]{PART1}). Indeed, $C\simeq\BP^1$ embeds into $\BP^2$ via a complete linear system of the line bundle $\CO(2)$. This gives an interpretation of the extra factor $\frac{1}{2}$ in Theorem \ref{thm:mainprimitive} and Proposition \ref{prop:CNC}.\footnote{I'm very grateful to Daniel Loughran for helpful discussion about this fact during the conference Heath-Brown@70 at University of Oxford.}
\end{remark}

\section{Perspectives on the Brauer--Manin obstruction}\label{se:BM}
 The projective conic $C=(F=0)\subset\BP^2$ being isomorphic to $\BP^1$ (we always assume that $C(\BQ)\neq\varnothing$), the universal torsor (up to isomorphism) over $C$ is $\BA^2\setminus\boldsymbol{0}$, as a  $\BG_{\operatorname{m}}$-torsor. The punctured affine cone $W^o$ is an intermediate $\BG_{\operatorname{m}}$-torsor over $C$, whose class in $H^1(C,\BG_{\operatorname{m}})$ has index two. So there exists a degree two map $$\pi:\BA^2\setminus\boldsymbol{0}\to W^o.$$ The Brauer group of $W^o$ (modulo the constant part) is isomorphic to $H^1(\BQ,\BZ/2\BZ)\simeq \BQ^*/\BQ^{*2}$, which can be identified with the set of all field extensions of degree two over $\BQ$. To parametrise all points of $W^o$ it is necessary to twist the map $\pi$ by elements in $\BQ^*/\BQ^{*2}$, and we denoted the twisted map by $$\pi^a:(\BA^2\setminus\boldsymbol{0})^a\to W^o$$ for $a\in\BQ^*/\BQ^{*2}$.
 Having fixed integral models for each, the diagram above extends to a $\BZ$-map $$(\BA_{\BZ}^2\setminus\boldsymbol{0})^a\to \CW^o\to \mathcal{C}\simeq \BP^1_{\BZ}.$$ Then by the Borel--Serre theorem (cf. e.g. \cite[Proposition 6.3]{C-D-X}), the number of $a\in\BQ^*/\BQ^{*2}$ such that $\pi^a((\BA_{\BZ}^2\setminus\boldsymbol{0})^a(\widehat{\BZ}))\cap \CD_{\CW^o}(L,\bGamma)\neq\varnothing$ is finite (we recall $\CD_{\CW^o}(L,\bGamma)\subset\CW^o(\widehat{\BZ})$ is the closed subset prescribed by the local condition $(L,\bGamma)$). This corresponds to the fact that for each fixed integral model, there are only finitely elements in the Brauer group which are ``active''.

Let us compare this with our analytic result Theorem \ref{thm:mainprimitive}. Although we are unable to compute the precise value of the constant $\CG_{L,\bGamma}(w)$, it features the values at $s=1$ of the $L$-functions associated to non-trivial real characters, which do not vanish. Note that there are only finitely many characters which appear once the adelic neighbourhood is fixed. They should be viewed as ``effects'' of the Brauer--Manin obstruction on the open adelic neighbourhood regarding the real condition prescribed by $w$ and the congruence condition $(L,\bGamma)$.\footnote{I'm very grateful to Yang Cao for generously sharing ideas for the materials in this section.}

\section*{Acknowledgements} 
The author was partially supported by National Key R\&D Program of China No. 2024YFA1014600.

\end{document}